\documentclass[12pt]{amsart}

\usepackage{amssymb}
\usepackage{amstext}
\usepackage{amsmath}
\usepackage{amscd}
\usepackage{latexsym}
\usepackage{amsfonts}
\usepackage{comment}
\usepackage{color}
\usepackage{enumerate}
\usepackage[all]{xy}
\usepackage{graphicx}

\theoremstyle{plain}
\newtheorem{thm}{Theorem}[section]
\newtheorem*{thm*}{Theorem}
\newtheorem*{cor*}{Corollary}
\newtheorem*{defn*}{Definition}

\newtheorem{prop}[thm]{Proposition}
\newtheorem{prob}[thm]{Problem}
\newtheorem{lem}[thm]{Lemma}
\newtheorem{cor}[thm]{Corollary}
\newtheorem{claim}[thm]{Claim}
\newtheorem*{claim*}{Claim}

\usepackage[pagebackref]{hyperref}
\usepackage{hyperref}
\hypersetup{
    colorlinks=true, 
    linkcolor=red, 
, 
}

\usepackage[alphabetic,initials]{amsrefs}

\theoremstyle{definition}
\newtheorem{defn}[thm]{Definition}

\newtheorem{rem}[thm]{Remark}

\theoremstyle{remark}

\begin{document}
\title[]{Flexibility of Affine cones over Mukai fourfolds of genus $g\ge7$}

\author[M. Hoff]{Michael Hoff}
\email{hahn@math.uni-sb.de}

\author[H.L. Truong]{Hoang Le Truong}
\address{Institute of Mathematics, VAST, 18 Hoang Quoc Viet Road, 10307
Hanoi, Viet Nam}
\address{Thang Long Institute of Mathematics and Applied Sciences, Hanoi, Vietnam}
\email{hltruong@math.ac.vn\\
	truonghoangle@gmail.com}

\thanks{
The first author was partially supported by the Deutsche Forschungsgemeinschaft (DFG, German Research Foundation) - Project-ID {286237555} - TRR 195.
The second author was partially supported by a fund of Vietnam National Foundation for Science
and Technology Development (NAFOSTED) and VAST under grant number NVCC01.04/22-23}

\begin{abstract}  
We show that the affine cones over a general Fano-Mukai fourfold of genus  $g=7$, $8$ and $9$ are flexible. Equivalently, there is an infinitely transitive action of the special automorphism group on such affine cones. In particular, any Mukai fourfold of genus $7,8$ and $9$ is $\Bbb A^2$-cylindrical. 
\end{abstract}

\maketitle

\section{Introduction}

Studying affine varieties and their automorphism groups is an active area of research. A remarkable property of an affine variety is the infinitely transitive action of the special automorphism group on it. Let $X$ be an affine variety over $\Bbb C$. 
A one-parameter subgroup $H$ of ${\rm Aut}(X)$ isomorphic as an algebraic group to the additive group $\Bbb G_a$ is called a unipotent one-parameter subgroup. We denote by ${\rm SAut}(X)$ the special automorphism group, that is, the subgroup of ${\rm Aut}(X)$ generated by all the one-parameter unipotent subgroups. 
Following \cite{AFKKZ13}, the variety $X$ is called {\it flexible} if ${\rm SAut}(X)$ acts highly transitively on the smooth locus ${\rm reg}(X)$, 
that is, it acts $m$-transitively for any natural number $m$.   

        In dimension $2$, del Pezzo surfaces of degree $\ge 4$ with their pluri-anticanonical polarizations have flexible affine cones (\cite{Per13,PaW16}).
		In higher dimensions, the affine cones over flag varieties of dimension $\ge 2$ are flexible by \cite{AZK12}. 
		In \cite{MPS18}, it is shown that the affine cone over the secant (tangential) variety of Segre-Veronese varieties are  flexible. 
		Furthermore, there are several constructions producing new flexible varieties from a given flexible one, e.g. via suspensions or open subsets with complement of codimension $\ge 2$ (see \cite{AFKKZ13,AZK12,FKZ16,KaZ99}).
		Furthermore, several families of examples of Fano varieties of dimension $3$ and $4$ and Picard number one admitting flexible affine cones have been constructed \cite{AZK12,APS14,MPS18,PrZ18,PrZ20}, 
		but a complete classification is still far from being known. 
		
In this paper we focus on the case in which $X$ is a Fano--Mukai fourfold, namely a four dimensional complex smooth Fano variety of index $2$. Given a smooth Fano fourfold $X$ with Picard rank $1$, the {\it index} of $X$ is the integer $r$ such that $-K_X = rH$, where $H$ is the ample divisor generating the Picard group: ${\mathrm{Pic}}(X) = \Bbb Z H$. 
The {\it degree} $d = \deg X$ is the degree with respect to $H$. It is known that $1 \le r \le 5$ and smooth Fano fourfolds of index $r = 2$ are called {\it Mukai fourfolds}; their degrees are even and can be written as $d = 2g - 2$, where $g$ is called the {\it genus} of $X$.

Our main theorem is the following. 

\begin{thm}\label{general1}
The affine cones over  general Mukai fourfold of genus  $g=7$, $8$ and $9$ are flexible.
\end{thm}

The only known flexible Mukai fourfolds are of genus $10$. 
According to  \cite{PrZ18,PrZ20}  a smooth del Pezzo fourfold $W_5 \subset \Bbb P^7$ of degree $5$ has flexible affine cones. Starting with the del Pezzo quintic fourfold $W_5$,  and performing suitable Sarkisov links, Prokhorov and  Zaidenberg  
showed that the affine cones over any  Mukai fourfold of genus $10$ are flexible. We proceed in a similar fashion starting with $\Bbb P^4$, but also use classical adjunction theory to prove our main theorem. 

Furthermore, the proof of Theorem \ref{general1} makes substantial use of a geometric property called $\Bbb A^n$-cylindricity, which is worthwhile to study for its own sake and is closely related to the property of being flexible (see \cite{Per13, PrZ20}). Recall that a smooth projective variety $X$ over $\Bbb C$ is called {\it $\Bbb A^n$-cylindrical} if it contains an $\Bbb A^n$-cylinder, that is, a principal Zariski open subset $U$ isomorphic to a product $Z \times \Bbb A^n$, where $Z$ is a variety.
By \cite{PrZ18}, smooth Fano-Mukai fourfolds of genus $10$ are $\Bbb A^4$-cylindrical; by \cite{HHT22}, a smooth Mukai fourfolds of genus $8$ is  $\Bbb A^1$-cylindrical; and there exist several families of codimension at least one of $\Bbb A^1$-cylindrical Mukai varieties of genus $g=7,9$ (see \cite{HHT22,PrZ16,PrZ17}).  We complete the picture of cylindricity for Mukai fourfolds of genus $\ge 7$.

\begin{thm}\label{thmCylinder}
 Any Mukai fourfold of genus $7,8$ and $9$ is $\Bbb A^2$-cylindrical. 
\end{thm}
In \cite{CPW16}, it is shown that the affine cone over any plurianticanonically polarized del Pezzo surface of degree 3 admits no $\Bbb G_a$-action. This leads to the notion of generic flexibility of affine varieties\footnote{That is, the special automorphism group acts infinitely transitive on an open orbit of the variety.} which has attracted attention quite recently (see \cite{AKZ19, Per21, KimPark21}).  In Section \ref{sectionGenericFlexibility}, we can conclude that any Mukai fourfold is generically flexible.

Finally, we refer to \cite{CPPZ21} for an excellent overview and introduction to cylindrical varieties and applications to flexible varieties.

\paragraph{\bf Outline.} In Section 2, we recall  results from adjunction theory that we used throughout the paper. Moreover, we describe  birational maps to Mukai fourfolds of genus $7, 8, 9,$ and $10$ based on a single blow up of certain Fano fourfolds along smooth surfaces and an adjunction map on the blow up. The third section is devoted to  the inverses of these maps. Cylinders are discussed in Section 4, where the proof of Theorem \ref{thmCylinder} (= Corollary \ref{MukaiFourfoldA2cylinder}) is given.
Section 5 contains the main result of this paper, Theorem \ref{general1} (= Theorem \ref{general}) and its proof. We close with the problem to generalize Theorem \ref{general1} to any Mukai fourfold.

\paragraph{\bf Notation and conventions.} Throughout the paper, we work over the complex numbers $\Bbb C$ and $X_{2g-2}$ is a Mukai fourfold of index $2$ and genus $g$. 
	 Let $Y\subset\Bbb P^N$ be a smooth projective variety of dimension $n$ and  let $d$ be its degree. Then $H_Y$ is its the hyperplane class  and $K_Y$ is its canonical class. 
The sectional genus of $Y$, denoted $\pi$, is the genus of the curve $Y\cap H_1\cap\ldots\cap H_{n-1}$, where $H_1, \ldots , H_{n-1}$ are generic hyperplanes in $\Bbb P^N$.
The adjunction formula reads $2\pi-2=H_Y^{n-1}((n-1)H_Y+K_Y)$.

\section{Birational maps to Mukai fourfolds and Adjunction theory}
\subsection{Adjunction theory}
We refer to \cite{BeS95} for a detailed history of results about adjunction theory.

\begin{thm}[{\cite[(0.1) Theorem.]{SoV87}}] \label{adj1}Let $\overline{X}$ be a complex projective manifold of dimension $n \ge 2$ and let $\overline{L}$ be a very ample line bundle over $\overline{X}$. Then $K_{\overline X}+ (n-1)\overline{L}$ is spanned by global sections unless either
\begin{enumerate}[$i)$]

\item $(\overline{X},\overline{L}) \cong (\Bbb P^n,\mathcal O_{\Bbb P^n}(1))$ or $(\Bbb P^2,\mathcal O_{\Bbb P^n}(2))$;
\item $(\overline{X},\overline{L}) \cong (Q^n,\mathcal O_{Q^n}(1))$ where $Q^n$ is a smooth quadric in $\Bbb P^{n+1}$
\item  $\overline{X}$ is a $\Bbb P^{n-1}$-bundle over a smooth curve and the restriction of $\overline{L}$ to a fibre is $\mathcal O_{\Bbb P^{n-1}}(1)$.
\end{enumerate}
\end{thm}
Suppose that $K_{\overline X}+ (n-1)\overline{L}$ is spanned by global sections. Then 
We write $\Phi_{K_{\overline{X}}+(n-1)\overline{L}}=r\circ \varphi$ for  the Stein factorization of $\Phi_{K_{\overline{X}}+(n-1)\overline{L}}$, so $\varphi:\overline{X}\to X$ is morphism with connected fibers onto a normal variety $X$ and $r$ is a finite map.

\begin{thm}[{\cite[(0.3) Theorem]{SoV87} or \cite[section 7.3]{BeS95}}] \label{adj2} 
Let $\overline{X}$ be a complex projective manifold of dimension $n \ge 2$ and let $\overline{L}$ be a very ample line bundle over $\overline{X}$. Assume that $K_{\overline X}+ (n-1)\overline{L}$ is spanned by global sections and the adjunction map has $n$-dimensional
image. Then $X$ is smooth and $\varphi:\overline{X}\to X$ expresses $\overline X$ as the blowing up of a finite set $s$ on a projective variety $X$.  If $L=\mathcal{O}_X(\varphi(\overline L))$ then  $L$  and $K_X+(n-1)L$ are ample and
$$\varphi^\ast(K_X+(n-1)L)=K_{\overline{X}}+(n-1)\overline{L}.$$
\end{thm}

\subsection{Birational maps to Mukai fourfolds}

We will describe birational maps to Mukai fourfolds of genus $g\in \{7, 8, 9, 10\}$ based on a single blow up of certain Fano fourfolds along smooth surfaces and an adjunction map on the blow up. The numerical data for our constructions are collected in the following table. 

\begin{table}[h!]
\centering
		\begin{tabular}{ |c| c| c|| c|c|c||c|c|c|c|}
	\hline
	 $Y$& $i(Y)$ &$d(Y)$&$F$&$d(F)$&$\pi(F)$& $g$\\ 
	\hline \hline
	$\Bbb P^4$&	$5$&$1$&linkage of a Veronese surface in $\Bbb P^4$&$8$&$6$&$7$\\ 
	\hline
	$Q^4\subset\Bbb P^5$&$4$&$2$&linkage of a Veronese surface in $\Bbb P^5$&$8$&$4$&$8$\\ 
	\hline
	$X_{2\cdot2}\subset \Bbb P^6$&3&$4$&Veronese surface in $\Bbb P^5$&$4$&$0$&$9$\\ 
	\hline
	$X_{1^3}\subset\Bbb G(1,4)$&$3$&$5$&sextic  del Pezzo in $\Bbb P^6$&$6$&$1$&$10$\\ 
	\hline
	
\end{tabular}
\caption{Fano fourfolds $Y$ containing surfaces $F$ which are mapped birationally to Mukai fourfolds of genus $g$.}
\label{S40}
\end{table}

Using the Riemann-Roch theorem and Kodaira vanishing theorem,  the surfaces $F\subset Y$ in Table \ref{S40} are contained in a unique hypersurface of degree $i(Y)-2$.  Moreover,  by linkage of a Veronese surface in $\Bbb P^4$ or $\Bbb P^5$, we have
the following Betti tables for the surfaces $F$ as in Table \ref{S40}. 
	\begin{center}
	\begin{align*}
	g=7: \quad
		\begin{tabular}{ c | c c c c c c}
			&   & $0$ & $1 $ &$2 $  &$3 $\\ 
			\hline
			0&  & $1    $ & $\cdot$ & $\cdot$ & $\cdot$ \\ 
			1&  & $\cdot$ & $\cdot$ &$\cdot$ &$\cdot$ \\ 
			2&  & $\cdot$ & $1$ &$\cdot$ &$\cdot$ \\ 
			3&  & $\cdot$ & $4   $ &$  5 $ &$  1  $ \\ 
		\end{tabular}
		\quad  g=8:		\quad
		\begin{tabular}{ c | c c c c c c}
			&   & $0$ & $1 $ &$2 $  &$3 $\\ 
			\hline
			0&  & $1    $ & $\cdot$ & $\cdot$ & $\cdot$ \\ 
			1&  & $\cdot$ & $2$ &$\cdot$ &$\cdot$ \\ 
			2&  & $\cdot$ & $4$ &$9$ &$4$ \\ 
		\end{tabular}		
		\end{align*}
		\end{center}
		\begin{center}
\begin{align*}
\quad  g=9:	\quad
		\begin{tabular}{ c | c c c c c c}
			&   & $0$ & $1 $ &$2 $  &$3 $&$4$\\
			\hline
			0&  & $1    $ & $1$ & $\cdot$ & $\cdot$& $\cdot$ \\ 
			1&  & $\cdot$ & $6$ &$14$ &$11$&$3$ \\ 
		\end{tabular}	
	\quad g=10: \quad
		\begin{tabular}{ c | c c c c c c c}
			&   & $0$ & $1 $ &$2 $  &$3 $&$4 $\\ 
			\hline
			0&  & $1    $ & $1$ & $\cdot$ & $\cdot$ &$\cdot$\\ 
			1&  & $\cdot$ & $9$ &$25$ &$25$ &$9$  \\ 
			2&  & $\cdot$ & $\cdot$ &$\cdot$ &$1$&$1$ \\ 
		\end{tabular}
		\end{align*}
		\end{center}

\begin{prop}\label{PropVeryAmple}
Let $Y$ be a smooth Fano fourfold of index $i:=i(Y)$ and degree $d(Y)$  in $\Bbb P^{g-3}$ as in Table \ref{S40} and let $H$ be an ample divisor on $Y$ whose class generates $\mathrm{Pic}(Y)$. Suppose that $Y$ contains a smooth surface $F$  as in Table \ref{S40}.  Let $g$ be an integer as in Table \ref{S40} depending on $Y$. Then, the linear system $|(i-1)H - F|$ of hypersurfaces of degree $i-1$ passing through $F$ defines a birational map 
$$\Phi_{(i-1)H-F}:Y \dashrightarrow \hat{X} \subset \Bbb P^{g+1},$$
where $\hat{X}$ is a fourfold of degree $2g-3$ and sectional genus $g$.  Moreover,  the fourfold $\hat{X}$ is \emph{isomorphic} to the blow up of $Y$ along $F$.
\end{prop}

\begin{proof}
 By Hironaka's theorem about elimination of indeterminacies, we can choose  a blowup $\sigma:\overline{X}\to Y$ of $F$ on $Y$ with exceptional divisor $E$ such that $\nu=\Phi_{(i-1) H - F}\circ \sigma:\overline{X}\to \hat{X}$ is a morphism.   We will show that the map $\nu$ is an isomorphism onto a smooth Fano fourfold $\hat{X}$.  We set $\overline{\mathcal L}=(i-1)\sigma^\ast H-E$ and will show that $\overline{\mathcal{L}}$ is very ample.  
 
 Since $F$ is a scheme-theoretical intersection of hypersurface of degree $i-1$, the linear system $|\overline{\mathcal L}|$ is base point free.  Hence,  the divisor 
 $$-K_{\overline{X}} = \sigma^\ast H +(i-1)\sigma^\ast H - E= \sigma^\ast H +\overline{\mathcal L}$$ 
is ample, that is, $\overline{X}$ is a Fano fourfold of index $i$ whose rank of the Picard group ${\mathrm{Pic}}(\overline X)$ is two.  Thus,  if $\nu$ is an isomorphism,  then the blowup $\overline{X}$ is mapped onto a smooth Fano fourfold $\hat{X}$ of index $i$, degree $2g-3$, and sectional genus $g$ (see Table \ref{intersectionNumber}).  

Let $D$ be  the proper transform of the unique hypersurface of $Y$ that passes through $F$,  and whose degree is $i-2$. We write $D \sim (i-2)\sigma^\ast H-kE$ for some $k > 0$.  Since $F$ is smooth,  we can compute $(\varphi^\ast H)^2\cdot E^2$,  $(\varphi^\ast H)\cdot E^3$,  and  
$E^4$,  and whence,  we have 
 $0\le \overline{\mathcal L}^3\cdot D$
as in Table \ref{intersectionNumber} by \cite[Lemma 2.3]{PrZ17}. It follows that $k=1$ and $\overline{\mathcal L}^3\cdot D=1$, and $\overline{\mathcal L}^2\cdot D^2=-1$, $\overline{\mathcal L}\cdot D^3=1$.  
\begin{table}[h!]
\centering
		\begin{tabular}{ |r| c| c| c|c|c|c|c|c|c|}
	\hline
	 $g$&$(\varphi^\ast H)^2\cdot E^2$&$(\varphi^\ast H)\cdot E^3$&$E^4$&$\overline{\mathcal L}^3\cdot D$&$D^4$&$\overline{\mathcal L}^4$\\ 
	\hline
	$7$&$-8$&$-42$&$-149$&$-29k+30$&$4$&$11$\\ 
	\hline
	$8$&$-8$&$-30$&$-77$&$-23k+24$&$3$&$13$\\ 
	\hline
	$9$&$-4$&$-6$&$-1$&$-13k+14$&$3$&$15$\\ 
	\hline
	$10$&$-6$&$-12$&$-15$&$-15k+16$&$2$&$17$\\ 
	\hline
	
\end{tabular}
\caption{Intersection numbers on blowup $Y$ along the surface $F$.}
\label{intersectionNumber}
\end{table}

We will show that the pair $(D,\overline{\mathcal L}_D)$ is isomorphic to $(\Bbb P^3,\mathcal O_{\Bbb P^3}(1))$. 
Note that $\overline{\mathcal{L}} - D \sim \sigma^*(H)$.  From the short exact sequence 
$$
0\to \overline{\mathcal L}-D \to \overline{\mathcal L}\to \overline{\mathcal L}_D\to 0
$$
we deduce that $h^0(\overline{\mathcal L}_D) =4$ since $h^1(\sigma^*H) = h^1(\sigma^*H - K_{\overline{X}} + K_{\overline{X}}) = 0$ by Kodaira vanishing. Furthermore,  since $\overline{\mathcal L}_D^3 = \overline{\mathcal L}^3\cdot D=1$,  the line bundle $\overline{\mathcal L}_D$ is ample and by \cite[Thm 3.1.2]{BeS95},  we conclude that  $(D,\overline{\mathcal L}_D)\cong(\Bbb P^3,\mathcal O_{\Bbb P^3}(1))$.

Let us now assume that $\overline{\mathcal L} $ is not  ample.  Then, $\nu$ is a birational morphism which contracts some curve. By the cone theorem and since $\overline{X}$ is a Fano variety,  we know that there exists an extremal ray $R$ such that $\overline{\mathcal L} \cdot R=0$.  Since  $\overline{\mathcal L}$ is basepoint free, $|\overline{\mathcal L}|$ defines a birational morphism $\nu:\overline{X}\to\hat X$ which coincides actually with the extremal contraction associated to $R$. Furthermore, the rank of the Picard group $\mathrm{Pic}(\overline{X})$ is two,  and whence, ${\mathrm{Pic}}(\hat X) \cong \Bbb Z$ and we can write $\overline{\mathcal L} =\nu^\ast L_{\hat X}$,  where $L_{\hat X}$ is the ample generator of ${\mathrm{Pic}}(\hat X) \cong \Bbb Z$.  

The line bundle $\overline{\mathcal L}$ yields a supporting linear function for the extremal ray generated by the curves contained in the fibers of $\nu$.  Since $\overline{\mathcal L}^4=$ $2g-3$ $>0$,  the class $R$ is not nef.
We have $\sigma^\ast H\cdot R>0 $ because $K_{\overline X}\cdot R<0$ and  therefore,  $D\cdot R=\overline{\mathcal L} \cdot R - \sigma^\ast H\cdot R<0$. We conclude that $D\cdot C<0$ for any curve $C$ in $\overline X$ with $[C]\in R$ so that the divisor $D$ contains all curves on $\overline X$ such that $[C]\in R$.

We denote $E^\prime=\{x\in\overline{X}\mid \nu \text{ is not isomorphism at }x \}$ the locus of curves whose numerical classes are in $R$ that is contained in the divisor $D$.  Since $\overline{\mathcal L}^3 D=1$,  we get that  
$\dim E^\prime\le 2$ and that $\nu$ is small (elementary) contraction (note that $\dim E^\prime = 3$ implies $D=E^\prime$ and hence,  $\overline{\mathcal L}^3 D=0$, a contradiction).

Applying \cite[Thm 1.1]{Kaw89},  we have that $E^\prime$ is the disjoint union of irreducible components $E^\prime_i$ such that $E^\prime_i\cong\Bbb P^2$ and $\mathcal N_{E^\prime_i/\overline{X}}\cong \mathcal O_{\Bbb P^2}(-1)\oplus \mathcal O_{\Bbb P^2}(-1)$.  Therefore,  we have $\ell(R)=\min\{-K_{\overline{X}}\cdot C\ |\ C \text{ rational, } [C]\in R\} = 1$, and the restriction $\nu|_D$ of $\nu$ on $D$ is a Mori contraction.  By \cite[Thm 4.1.2]{AnM03},  we have $\nu|_D$ is the blow up of a smooth point in a smooth threefold. Thus, $\textrm{rank } \mathrm{ Pic}(D)\ge2$,  contradicting $(D,\overline{\mathcal L}_D)\cong(\Bbb P^3,\mathcal O_{\Bbb P^3}(1))$.  Finally, the line bundle $\overline{\mathcal L}$ is ample. Using the Riemann-Roch theorem and Kodaira vanishing theorem, we obtain the equality $\dim|\overline{\mathcal L}|=g+1$ and $\overline{\mathcal L}$ is very ample.
\end{proof}

\begin{prop}\label{main}
Let $Y$ be a smooth Fano fourfold of index $i:=i(Y)$ and degree $d(Y)$  
in $\Bbb P^{g-3}$ as in Table \ref{S40} and let $H$ be an ample divisor on $Y$ whose class generates $\mathrm{Pic}(Y)$. Suppose that $Y$ contains a smooth surface $F$  as in Table \ref{S40}.  Let $g$ be an integer as in Table \ref{S40} depending on $Y$. Then we have the following statements.

\begin{enumerate}[$i)$]

\item  The linear system $|(2i-3)H - 2F|$ of hypersurfaces of degree $2i-3$ with double points along the surface $F$ defines a birational map 
$$\Phi_{(2i-3)H-2F}:Y \dashrightarrow X \subset \Bbb P^{g+2},$$
where $X$ is a Mukai fourfold of  genus $g$ with ${\mathrm{Pic}}(\overline{X} ) \cong \Bbb Z$.  

\item There is a commutative diagram
$$\xymatrix{&E\ar@{^{(}->}[r]\ar[ld]&\overline{X}\ar[dl]_{\sigma}\ar[dr]^{\varphi}&D\ar@{_{(}->}[l]\ar[rd]&\\
F\ar@{^{(}->}[r]&Y\ar@{-->}[rr]^{\Phi_{(2i-3)H-2F}}&&X\subseteq\Bbb P^{g+2}&x_0\ar@{_{(}->}[l]
}$$
where $\sigma$ is the blowup of $F$ with exceptional divisor $E$. Let $\overline{\mathcal L}=(i-1)\sigma^\ast H- E$ be the very ample divisor defining the birational morphism to $\hat{X}$ as in Proposition \ref{PropVeryAmple}. The adjunction map $\Phi_{K_{\overline{X}}+3  \overline{\mathcal L}} = \varphi$ expresses $\overline X$ as the blowup of a point $x_0$ on a smooth Mukai fourfold $X$ of genus $g$ with exceptional divisor $D$.

\item Let  $L$ be the ample generator of ${\mathrm{Pic}}(X)$. Then we have the following relations 
$$\begin{aligned}
\varphi^\ast L &\sim& (2i-3)\sigma^\ast H- 2E,\quad\quad\quad\quad&\quad\quad D\sim (i-2)\sigma^\ast H-E,\\
\sigma^\ast H &\sim& \varphi^\ast L- 2D, \quad\quad\quad\quad&\quad\quad E\sim (i-2) \varphi^\ast L-(2i-3)D.
\end{aligned}
$$
\item $\sigma(D)$ is the unique hypersurface of degree $i-2$ containing $F$.

\item $Y\backslash \sigma(D)\cong X\backslash \varphi(E)$.

\end{enumerate}
\end{prop}

\begin{rem}
The base locus of $\Phi_{(2i-3)H-2F}$ is the union of the surface $F$ and $\ell_g$ planes,  where $\ell_g$ is the number of lines through a general point on a Mukai fourfolds $X_{2g-2}$ (see also Section \ref{lines}).  
\end{rem}

\begin{proof}

To $i)$ and $ii)$: Let $\overline{\mathcal L}=(i-1)\sigma^\ast H-E$ be the very ample line bundle on the smooth variety $\overline{X}$. 
Then by  Theorem \ref{adj1}, the line bundle $K_{\overline X}+3\overline{\mathcal L}$ is spanned by global sections and
$$K_{\overline{X}}+3\overline{\mathcal L}=-i\sigma^\ast H + E+3\cdot((i-1)\sigma^\ast H - E)=(2i-3)\sigma^\ast H - 2E.$$
Using the Riemann-Roch theorem and Kodaira vanishing theorem, we obtain that $\dim |K_{\overline{X}}+3 \overline{\mathcal L}| = g+2$. 
Now we write $\Phi_{K_{\overline{X}}+3\overline{\mathcal L}}=r\circ \varphi$ for the Stein factorization of $\Phi_{K_{\overline{X}}+3\overline{ \mathcal L}}$, so $\varphi:\overline{X}\to X$ is a morphism with connected fibers onto a normal variety $X$ and $r$ is a finite map.  This yields the diagram
$$\xymatrix{&\overline{X}\ar[dr]_{\Phi_{K_{\overline{X}}+3\overline{\mathcal L}}}\ar[dl]_{\sigma}\ar[r]^{\varphi}&X\ar[d]^r\\
Y\ar@{-->}[rr]_{\Psi_{(2i-3) H-2F}}&&X^\prime \subset\Bbb P^{g+2},}$$
where $\overline{X}\to X^\prime \subset\Bbb P^{g+2}$ is given by the linear system $|K_{\overline X}+3\overline{\mathcal L}|$. We will show that $r$ is an isomorphism.

Since $h^0(K_{\overline{X}}+2\overline{\mathcal L})=1>0$, the image of $\overline{X}$ is of dimension $\dim \Phi_{K_{\overline{X}}+3\overline{\mathcal L}}(\overline X)=4$.  It follow from Theorem \ref{adj2} that $X$ is smooth and $\varphi:\overline{X}\to X$ expresses $\overline X$ as the blowing up of a finite set $s$ on a projective variety $X$. 
Since ${\mathrm{rank}}\ {\mathrm{Pic}}(\overline{X})=2$, the set $s$ is a point $x_0$ and so $\varphi$ is the blowing up of the point $x_0$ on a smooth projective variety $X$.
Furthermore, we can write $K_{\overline{X}}+3\overline{\mathcal L}= \varphi^\ast L$, where $L$ is the ample generator of $ {\mathrm{Pic}}(X)\cong \Bbb Z$.
 
As in Theorem \ref{adj2}, we define $\mathcal L=\mathcal{O}_X(\varphi(\overline {\mathcal L}))$, and get that  $\mathcal L$  and $K_X+3\mathcal L$ are ample. Furthermore, 
$\overline {\mathcal L}=\varphi^\ast (\mathcal L)-\mathcal{O}(\varphi^{-1}(x_0))$. 
Since $|K_{\overline{X}}+3\overline{\mathcal L}|=|(2i-3)\sigma^\ast H - 2E|$ defines a birational morphism $\overline X\to X$, which coincides actually with the map $\varphi$, we have
$$\mathcal{O}(\varphi^{-1}(x_0))=\varphi^\ast (\mathcal L)-\overline {\mathcal L}=((2i-3)\sigma^\ast H - 2E)-((i-1)\sigma^\ast H - E)=(i-2)\sigma^\ast H - E \sim D.$$
Thus, $\varphi$ is birational and its exceptional locus coincides with $D$. In particular, it is an irreducible divisor. 

 Since $K_{\overline{X}}=-i\sigma^\ast H +E$,  $\overline{\mathcal L}=(i-1)\sigma^\ast H -E$ and $D=(i-2)\sigma^\ast H -E$, we get that $$K_{\overline{X}} = -2(K_{\overline{X}}+3\overline{\mathcal L})+3D,$$ and $-K_{X}=2L$. Hence, $-K_{X}$ is an ample Cartier divisor divisible by $2$ in ${\mathrm{Pic}} X$. We conclude that $X$ is a Mukai fourfold.  The morphism $r:X\to X^\prime \subset \Bbb P^{g+2}$ is given by the linear system $|L| = | -\frac{1}{2} K_{X}|$. This is an isomorphism by \cite[Lemma 5.6]{PrZ17}.

 In the remaining proof, we identify $X$ with $X^\prime $ and $\varphi_{|3\rho^\ast H-2E|}$ with $\varphi$. Since $K_{\overline{X}}+3\overline{\mathcal L}=(2i-3)\sigma^\ast H - 2E$, we have $$2g(X)-2=(K_{\overline{X}}+3\overline{\mathcal L})\cdot \overline{\mathcal L}^3$$ and the sectional genus of $\overline{X}$   is $g$. Thus, 
$\varphi$ is the blowup of a point $x_0$ on  a smooth  Mukai fourfold $X$ of  genus $g$ with ${\mathrm{Pic}}(X ) \cong \Bbb Z$ and exceptional divisor $D$. This proves $i)$ and $ii)$.

To  $iii)$ and $iv)$: From the relations $E\sim (i-2) \varphi^\ast L-(2i-3)D$ and  $D\sim (i-2)\sigma^\ast H-E$ in ${\overline{X}}$, we deduce that the images $\varphi(E)$ and $\sigma(D)$ are hypersurface sections of $X$ and $Y$, respectively. Furthermore, $\varphi(E)$ is singular along $x_0$ and $\sigma(D)$ is the unique hypersurface section of  degree $i-2$ which contains $F$.

To  $v)$: Since $F \subset \sigma(D)$ we have isomorphisms
$$X\backslash \varphi(E)\cong {\overline {X}}\backslash(E\cup D)\cong {Y}\backslash(F\cup \sigma(D)) \cong Y\backslash \sigma(D).$$
\end{proof}


\section{Elementary rational maps with center at a point}\label{tangentialProjection}

In the previous section, we described birational maps to Mukai fourfolds of genus $7,8,9$ and $10$. We will show that the inverses of these maps are given by tangential projections from a general point of these Mukai fourfolds.  Therefore we describe lines on a Mukai fourfold passing through a fixed point and then analyze the tangential projection from this point. 

\subsection{Lines through a general point on a Mukai fourfolds}\label{lines}

Let $X_{2g-2}$ be a Mukai fourfold of genus $g\in \{7,8,9,10\}$ and let $p$ be a general point on $X_{2g-2}$. What is the intersection of $X_{2g-2}$ and its tangent space $T_p(X_{2g-2})$ at $p$? The answer to this question is well-known and follows from a result of Mukai and Landsberg--Manivel. 

In \cite{Muk89},  Mukai gave a description of smooth Mukai fourfolds of genus $g=7$, $8$, $9$, $10$  in terms of linear sections of appropriate rational homogeneous varieties $M_g^d$,  so-called Mukai models of sectional genus $g$ and dimension $d$ (see also the treatment in \cite{Kap18}).  An anti-canonical embedded Mukai fourfold $X_{2g-2}$ of genus $g$ is a complete intersection of $M_g^d$ and a $\Bbb P^{g+2}$.  The intersection of $X_{2g-2}$ and its tangent space $T_p(X_{2g-2})$ consists of $\ell_g$ lines for a general point $p$.  Indeed, this can be deduced from the fact that the intersection $\Bbb T_p\cap M_g^d$ of the Mukai model $M_g^d$ and its tangent space $\Bbb T_p$ at a point $p$ is a cone over some variety of degree $\ell_g$ which is of codimension $3$ in $\Bbb T_p$. Hence, the intersection $\Bbb P^{g+2}\cap \Bbb T_p\cap M_g^d$ is one-dimensional and consist exactly of $\ell_g$ lines.  In the following table, we may summarize the results of this section which follow from \cite[Theorem 4.8]{LM03}.  

\begin{table}[h!]
\centering
		\begin{tabular}{ |c| c| c| c|c|c|c|c|c|c|}
	\hline
	 $g$& $M_g^d$ & $M_g^d\cap \Bbb T_p$ & $\ell_g$ \\
	\hline
	$7$&  $M_7^{10}=\mathrm{OG}(5,10)$ & cone over $\Bbb {G}(1,4)\subset \Bbb P^9$ &5\\ 
	\hline
	$8$& $ M_8^{8}=\Bbb {G}(1,5)$  & cone over $\Bbb P^1\times \Bbb P^3$ &4	\\ 
	\hline
	$9$&$M_9^6=\mathrm{LG}(3,6)$ & cone over the Veronese surface &4\\ 
	\hline
	$10$&$M_{10}^5=\mathrm{G}_2$  & cone over twisted cubic curve &3\\ 
	\hline
	
\end{tabular}
\caption{ Mukai fourfolds of genus $g=7$, $8$, $9$, $10$.}
\label{Mukai}
\end{table}

\subsection{Terminology and tangential projections} Let $ X:=X_{2g-2} \subseteq \Bbb P^{g+2}$ be a Mukai fourfold of genus $g\in\{7$, $8$, $9$, $10\}$ with ${\mathrm{Pic}}(X ) \cong \Bbb Z L$, and let $L$ be the hyperplane class of $X$.
Let $x_0\in X$ be a general point on $X$, and let $\mathcal{L}_{x_0}$ be the union of lines passing through $x_0$. We denote $\ell_g$ the number of lines through a point. Then, we have two birational models of the blow up of $X$ in $x_0$ as follows:
\begin{enumerate}[$\bullet$]

\item Let $\varphi^\prime:X^\prime\to X$ be a blowup of $x_0$ on $X$ with exceptional divisor $D^\prime$. We use  $L^\prime$  to denote the pullback of the hyperplane class $L$ to $X^\prime$.
Let $\mathcal{L}^\prime_{x_0}$ be the proper transform of $\mathcal{L}_{x_0}$. Then,  $\mathcal{L}^\prime_{x_0}$  is the disjoint union of $\ell_g$ lines $\{\ell_i^\prime\}_{i=1}^{\ell_g}$ on $X^\prime$ and each $\ell_i^\prime$ cut $D^\prime$ along a point $x_i$.
\item We blow up $X^\prime$ along $\mathcal{L}^\prime_{x_0}$ to obtain $\widetilde{X}$ and introduce ${\ell_g}$ exceptional divisors $D_1,\ldots,D_{\ell_g}$ isomorphic to $\Bbb P^3$. 
 Let $\widetilde{D}$ be the proper transforms of $D^\prime$ and let $\widetilde{L}$ be the pullback  of $L^\prime$ to $\widetilde{X}$. Then $\widetilde{D}$ is isomorphic to $\Bbb P^3$ blown up along $\ell_g$ points $\{x_i\}_{i=1}^{\ell_g}\subset D^\prime$ with exceptional divisors $\{\widetilde{P_i}\}_{i=1}^{\ell_g}$. Thus, each $D_i$ cut $\widetilde{D}$ along a plane $\widetilde{P_i}\cong \Bbb P^2$. 
  \end{enumerate}
 We have the following intersection numbers.
 
\begin{lem}\label{inter}
We have $\widetilde{L}^4=2g-2$,  $\widetilde{L}\widetilde{D}=\widetilde{L}D_i=0$, $D_i^4=1$, $\widetilde{D}^4={\ell_g}-1$, $\widetilde{D}^3D_i=\widetilde{D}D_i^3=-1$, $\widetilde{D}^2D_i^2=1$ for all $i=1,\ldots,4$. 
\end{lem}
\begin{proof}
First, $\widetilde{L}D=\widetilde{L}D_i=0$ since their intersection is empty. It's clear that $\widetilde{L}^4=2g-2$.\\
Notice that $\widetilde{D}$ is isomorphic to $D^\prime\cong\Bbb P^3$ blown up along the points $\{x_i\}_{i=1}^{\ell_g}$ with the exceptional divisors $\{\widetilde{P_i}\}_{i=1}^{\ell_g}$. Thus each $D_i$ cut $\widetilde{D}$ along a plane $\widetilde{P_i}\cong \Bbb P^2$. Write $\mathrm{Pic} (\widetilde{D})=\langle \widetilde{H},\widetilde{P_1},\ldots,\widetilde{P_{\ell_g}}\rangle$ where $\widetilde{H}$ is the polarization from $\Bbb P^3$ while $\widetilde{P_i}$ are the exceptional divisors over the points $x_i$. Since $N_{x_i/\Bbb P^3}=\mathcal O_{x_i}^3(1)$ then write $\xi=c_1(\mathcal O_{\Bbb P(N_{x_i/\Bbb P^3})}(-1))$ and equality $\xi=\widetilde{P_i}|_{\widetilde P_i}$ we get that $\xi^2=1$ in the Chow group of $\widetilde{P_i}=\Bbb P(N_{x_i/\Bbb P^3})$. We have $N_{\widetilde{D}/\widetilde{X}}=\mathcal O(\widetilde{H})$ and $D_i|_{\widetilde{D}}=-\widetilde{P}_i$. Thus we obtain 
$$ \begin{aligned}
D_i^3\widetilde{D}&=(-\widetilde{P}_i)^3=-1,&
D_i^2\widetilde{D}^2&=(-\widetilde{P}_i)^2H=1,&
\text{and}&&D_i\widetilde{D}^3&=(-\widetilde{P}_i)H^2=-1.\\
\end{aligned}
$$
Now, since $\varphi^\prime:X^\prime\to X$ is a blowup of $x_0$ on $X$ with exceptional divisor $D^\prime$, we have $\rho^{\ast}(D')=\widetilde{D}+\sum\limits_{i=1}^{\ell_g}D_i$ and so that $\widetilde{D}^4={\ell_g}-1$ because $D'^4=(-1)^{4-1}$. It follows from $K_{X^\prime}=\varphi'^{\ast}K_X-3D^\prime$ and \cite[section 15.4]{Ful98} that $D_i^4=1$.
\end{proof}

We now consider the morphism induced by the linear system $|\widetilde{L}- \sum_{i=1}^{\ell_g}D_i|$. By Lemma \ref{inter}, this morphism contracts each $D_i$ to a $\Bbb P^2$. Geometrically,  each $D_i$ is a $\Bbb P^1$-bundle over $\Bbb P^2$. Indeed, fix a general point $p_i\in \widetilde{P_i}$ and  through each $p_i\neq p^\prime_i\in D_i$  passes a unique line $l$ intersecting $P_i$. The bundle map is given by $p^\prime_i\mapsto l\cap \widetilde{P_i}$. The blow down of each $D_i$ to $\Bbb P^2$ is given by the above linear system. The resulting $\overline{X}$ is birational to $X' = \textrm{Bl}_{x_0}(X)$. Then we organize these maps into a diagram:
\begin{align}\label{sequenceOfBlowUps}
\xymatrix{&\widetilde{X}\ar[dl]_{\rho^\prime}\ar[dr]^{\rho}&&&\\
X^\prime\ar[dr]^{\varphi^\prime}&&\overline{X}\ar[dl]_{\varphi}&&\\
&X&&&
}\end{align}
where $\varphi: \overline X\to X$ is the another blow up of $X$ along $x_0$ with the exceptional divisor $D$. 
Note that each plane $P_i$ for $i=1,\ldots,{\ell_g}$ is contained in $D$ and therefore,  contracted under $\varphi$. 

In the following, we consider a tangential projection of a Mukai fourfold from a general point $x_0$ on it such that there are exactly $\ell_g$ lines through $x_0$ as in Section \ref{lines}.  On the blow up $\widetilde{X}$, the tangential projection is given by the linear system $$|\widetilde{L} - \sum_{i=1}^{\ell_g}D_i -2\widetilde{D}|.$$

\begin{prop}\label{invMain}
Let $ X \subseteq \Bbb P^{g+2}$ be a Mukai fourfold of genus $g=7$, $8$, $9$, $10$ with ${\mathrm{Pic}}(X ) \cong \Bbb Z L$, and let $L$ be the hyperplane class of $X$.
Let $x_0\in X$ be a general point, and let $\mathcal{L}_{x_0}$ be the union of  lines passing through $x_0$. Then we have the following statements. 
\begin{enumerate}[$i)$]
\item The tangential projection from the point $x_0$ defines a birational map 
$$\Phi_{x_0}:X \dashrightarrow Y,$$
where $Y$ is a smooth Fano variety of index $i:=i(Y)$ and degree $d(Y)$  in $\Bbb P^{g-3}$ as in Table \ref{S40}. 
 Let $\overline{X}$ be in (\ref{sequenceOfBlowUps}), that is, the blow up of $X$ in $x_0$ with exceptional divisor $D$. \color{black} There is a commutative diagram
$$\xymatrix{&D\ar@{^{(}->}[r]\ar[ld]&\overline{X}\ar[dl]_{\varphi}\ar[dr]^{\sigma}&E\ar@{_{(}->}[l]\ar[rd]&\\
x_0\ar@{^{(}->}[r]&X\ar@{-->}[rr]^{\Phi_{x_0}}&&Y&F\ar@{_{(}->}[l]
}$$
where $\sigma$ is the blowup of a smooth surface $F$ as  in Table \ref{S40}  with exceptional divisor $E$.

\item Let  $H$ be the ample generator of ${\mathrm{Pic}}(Y)$. Then we have the following relations 
$$\begin{aligned}
\varphi^\ast L &\sim& (2i-3)\sigma^\ast H- 2E,\quad\quad\quad\quad&\quad\quad D\sim (i-2)\sigma^\ast H-E,\\
\sigma^\ast H &\sim& \varphi^\ast L- 2D, \quad\quad\quad\quad&\quad\quad E\sim (i-2) \varphi^\ast L-(2i-3)D.
\end{aligned}
$$

\item $X\backslash \varphi(E)\cong Y\backslash \sigma(D)$.
\item $\sigma(D)$ is a singular hypersurface section of $Y$ with its degree $i-2$ and its singular locus coincides with the locus of points $p \in \sigma(D)$, such that the restriction $\sigma|_D:D\to \sigma(D)$ is not an isomorphism over $p$.
\item Let $g=7$. Then $\varphi(E)$  coincides with the union of conic curves in $X$ passing through $x_0$ and for all  conic curves $C$ passing through $x_0$,  $C\subset \varphi(E)$.

\end{enumerate}
 In particular, the rational map $\Phi_{x_0}$ is the inverse of the rational map $\Phi_{(2i-3)H-2F}:Y \dashrightarrow X \subset \Bbb P^{g+2}$ as in Proposition \ref{main}.
\end{prop}

\begin{proof}

To $i)$ and $ii)$: 
Recall that $\widetilde{L}-\sum\limits_{i=1}^{\ell_g} D_i$ on $\widetilde{X}$ defines a birational morphism $\widetilde{X} \to\overline X$ which coincides with the map $\rho$. It is clear that $$-K_{\overline X}=- (\varphi^\prime\circ \rho^\prime)^\ast K_X+3D+2\sum\limits_{i=1}^4 D_i =-2(L-\sum\limits_{i=1}^4 D_i)+3D.$$
Thus we have $K_{\overline X}=-2\varphi^\ast L$, that is, $\overline{X}$ is a Fano fourfold whose rank of ${\mathrm{Pic}}(\overline{X})$ is two.  
By the Cone theorem, there exists a Mori contraction $\sigma: \overline{X} \to U$ different from $\varphi$.
We organize these maps into a diagram:
$$\xymatrix{&\widetilde{X}\ar[dl]_{\rho^\prime}\ar[dr]^{\rho}&&&\\
X^\prime\ar[dr]^{\varphi^\prime}&&\overline{X}\ar[dl]_{\varphi}\ar[dr]^{\sigma}&&\\
&X\ar@{-->}[rr]^{\Phi_{x_0}}&&U&
}$$

By Lemma \ref{inter}, we have $(\varphi^\ast L)^4=2g-2+\ell_g$, and $(\varphi^\ast L)^3\cdot D=\ell_g$, $(\varphi^\ast L)^2\cdot D^2=\ell_g$, $(\varphi^\ast L)\cdot D^3=\ell_g$ and $D^4=\ell_g-1$. Therefore, we have  $(\varphi^\ast L-2D)^3\cdot ( (i-2) \varphi^\ast L-(2i-3)D)=0$ and  $(\varphi^\ast L-2D)^4 = d(Y)$. 
This means that the divisor class of $\varphi^\ast L-2D$ is not ample, and so it yields a supporting linear function of the extremal ray generated by the curves in the fibers of $\varphi$. Moreover, we can write $\varphi^\ast L-2D=\sigma^\ast H$, where $H$ is the ample generator of ${\mathrm{Pic}}(U) \cong \Bbb Z$. 
By the Riemann-Roch theorem and Kodaira vanishing theorem, we have $\dim |\varphi^\ast L-2D|= 4$.  Thus, 
$\varphi^\ast L-2D$ defines a birational morphism $\overline{X} \to Y$ which coincides with the map $\sigma$. The birationality of $\Phi_{x_0}$ follows. This prove $i)$.

 Since $\dim |\varphi^\ast L-3D| = 0$ and $(\varphi^\ast L-2D)^3\cdot ( (i-2) \varphi^\ast L-(2i-3)D)=0$, the linear system 
$|(i-2) \varphi^\ast L-(2i-3)D)|$ contains a unique divisor $E$ contracted by $\sigma$. Since rank of ${\mathrm{Pic}}(\overline X)$ is two, the
divisor $E$ is irreducible and the relations in $iii)$ follow.

Since $(\varphi^\ast L-2D)^2\cdot E^2 = -d(F)$, the image $F=\sigma(E)$ is a {\it surface} with $\deg F = H^2\cdot \sigma(E) = d(F)$ (see Table \ref{S40}). Furthermore, the sectional genus of $F$ as given in Table \ref{S40} can be computed using 
 \cite[Lemma 2.3]{PrZ17} or \cite[Lemma 2.2.14]{IsP99}.

Since ${\mathrm{rank}}\ {\mathrm{Pic}} (\overline{X})= 2$, the exceptional locus of $\sigma$ coincides with $E$, and
$E$ is a prime divisor. Therefore, $\sigma$ has at most a finite number of $2$-dimensional fibers. By the main theorem in \cite{AnW98}, $F$ has at most isolated singularities. Furthermore, since $E^4=(\varphi^\ast L-3D)^4=-1$
and the classification of
surfaces of degree $\ge 8$  implies that  $F$ is a smooth surface as in Table \ref{S40} contained in  $\sigma(D)$.
Hence the morphism $\sigma : \overline{X}\to Y$ is the blowup of the {\it smooth} surface $F$, as stated in $iii)$.

To $iii)$:  Since $F \subset \sigma(D)$ we have isomorphisms
$$X\backslash \varphi(E)\cong {\overline X}\backslash(E\cup D)\cong {Y}\backslash(F\cup \sigma(D)) \cong Y\backslash \sigma(D).$$

To $iv)$: From the relations $E\sim (i-2) \varphi^\ast L-(2i-3)D$ and  $D\sim (i-2)\sigma^\ast H-E$ in $\overline{X}$, we deduce that the images $\varphi(E)$ and $\rho(D)$ are a hyperplane sections of $X$ and $Y$, respectively. The hypersurface $\sigma(D)$ of degree $i-2$ is singular because it contains the surface
$F$.  Since $-K_D \sim 2(\varphi^\ast L - 2D)|_D \sim \sigma^*(2H)|_D$,  the restriction $\rho|_D : D \to \rho(D)$ is a crepant morphism. 

To $v)$:   The linear projection $\Phi_{x_0} : X \dashrightarrow Y$ sends $\varphi(E)$ to a smooth surface $F$ in Y contracting the conic curve passing through $x_0$ in $\varphi(E)$. These conics correspond to the rulings of $\sigma|_E:E \to F$. It follows that $\varphi(E)$ is covered by such conics. Since through any point in $\varphi(E)$ passes a unique ruling of $\sigma|_E:E \to F$, given a point $x_0\neq p\in \varphi(E)$  there is a unique conic in $\varphi(E)$ through $p$ and $x_0$.

Any curve of degree $d$ through the point $x_0$ will be either tangent to the tangent space at $q_0$ or will be contained in it (as the five lines). In the first case, this means that the intersection multiplicity is at least $2$ and the projected curve has degree at most $d-2$. Hence, the image of a conic curve passing through $x_0$ has degree $0$ which means it is contracted to a point. For the second case, since $x_0$ is a general point, intersection of the tangent space of $X$ at $x_0$ with X is the five lines. Thus  a conic curve passing through $x_0$ belong in $\varphi(E)$.

\end{proof}

\section{$\Bbb A^2$-cylinders on Mukai fourfolds of $g= 7$, $8$, $9$ }

In this section, we construct an open subset of any Mukai fourfold $X_{2g-2}$ of genus $g\in\{7,8,9\}$ that can be covered by $\Bbb A^2$-cylinders. Let us recall the definition of $\Bbb A^n$-cylinder.

\begin{defn}
 An {\it $\Bbb A^n$-cylinder} in $X$ is a pair $(Z, \varphi)$  and an open embedding $\varphi :Z \times \Bbb A^n \to X$ where $Z$ is a variety and $\Bbb A^n$ is the affine $n$-space over $\Bbb C$.
We say that $X$ is {\it $\Bbb A^n$-cylindrical} if there exists an $\Bbb A^n$-cylinder $(Z, \varphi)$ in $X$.  
Given a divisor $H\subset X$ we say that a $\Bbb A^n$-cylindrical $(Z, \varphi)$ in $X$ is \emph{$H$-polar} if $\varphi( Z\times \Bbb A^n)=X\backslash \textrm{supp} D$ for some effective divisor $D\in |kH|$ where $k>0$. (\cite[Definitions 3.1.5, 3.1.7]{KPZ11})
\end{defn}

In the first step, we show that for a Mukai fourfold of genus $g\in \{7,8,9\}$, there exists a hypersurface $Y\subset X_{2g-2}$ and a singular cubic $W$ such that $X_{2g-2}\backslash Y \cong \Bbb P^4\backslash W$.
The second step is to construct an explicit $\Bbb A^2$-cylinder covering of the complement of a nodal cubic threefold in $\Bbb P^4$. The proof follows ideas of \cite[Proposition 5.2]{HHT22}.

\begin{prop}\label{openA2cover}
Let $X_{2g-2}$ be a Mukai fourfold for $g\in \{7,8,9\}$.  
There exists an open subset $U$ of $X$ such that
\begin{enumerate}[$i)$]
\item $U=X_{2g-2}\backslash Y$, where $Y$ is a cubic hypersurface section of $X_{2g-2}$.
\item There exists  a singular cubic threefold $W\subset\Bbb P^4$,  and a birational map $\Phi:X_{2g-2} \dashrightarrow  \Bbb P^{4}$ such that  $$X_{2g-2}\backslash Y\cong \Bbb P^4\backslash W.$$
\end{enumerate}
\end{prop}

\begin{proof}
 {\it Case $g=7$}: Proposition \ref{invMain}. 
 Note that the hypersurface $Y$ is the unique cubic hypersurface that intersects $X_{12}$ in a point with multiplicity at least $7$. The cubic threefold $W$ is the unique cubic threefold with the maximal number of 10 nodes.  
 
 {\it Case $g=8$}: This case follows immediately from \cite[Proposition 3.1]{HHT22} and the cubic threefold $W$ is again the unique cubic threefold with the maximal number of 10 nodes.  Note that the base locus of the birational map from $\Bbb P^4$ to $X_{14}$ is a surface of degree $7$ and sectional genus $4$.
 
 {\it Case $g=9$}: This case follows from Proposition \ref{invMain} and \cite[Theorem 3.1]{PrZ16}. Indeed, by Proposition \ref{invMain} we find a hyperplane $H\subset \Bbb P^{11}$ and a hyperplane $H' \in \Bbb P^6$ such that $X_{16}\backslash H \cong X_{2\cdot 2}\backslash H'$.  In \cite[Theorem 3.1]{PrZ16} the author show that the complement of a quadric hypersurface section of a complete intersection of two quadrics $X_{2\cdot 2}\subset \Bbb P^6$ is isomorphic to the complement of a quadric hypersurface in $\Bbb P^4$. Therefore,  the hypersurfaces $Y$ and $W$ as stated in the proposition are the union of a hyperplane and a quadric hypersurface.  
\end{proof}

\begin{prop}\label{cylinderCremona}
Let $W$ be a singular cubic hypersurface in $\Bbb P^4$ and let $p$ be an ordinary double point of $W$. 
Then for a point $x\in\Bbb P^4\backslash W$ there exists a hyperplane $P$, a quadric $Q$ and a principal affine open subset $U=U_x\subset P^4\backslash W$ such that 
\begin{enumerate}[$i)$]
\item $x\in U$;
\item $U\cong \Bbb A^2\times Z$, where $Z$ is an affine surface.
\item $U=\Bbb P^4\backslash (P\cup Q\cup W)$.
\end{enumerate}
\end{prop}

\begin{proof}
Step 1: 
Let $W = V(f)$ be generated by a cubic equation $f$, and let $\mathcal{I}_p(1) = (l_1,l_2,l_3,l_4)$ be the linear forms generating the ideal of $p$. We choose a rank $4$ quadric $Q=V(g)$ that is singular in $p$ (that is,  $g\in |\mathcal{I}^2_p(2)|$). Then, we consider the following cubic Cremona transformation 
$$
\xymatrix{
\Psi_Q: \Bbb P^4 \ar@{-->}[rrr]^{|g \cdot l_1, g \cdot l_2, g \cdot l_3, g \cdot l_4, f|} &&& \Bbb P^4.}
$$
The image of $W$ is a hyperplane $P_W$ of $\Bbb P^4$, and the image of $Q$ is point which does not lie on $P_W$.  
The total transform of the point $p$ is another rank $4$ quadric $Q'$ whose vertex is the image of $Q$.  Note that the inverse of $\Psi_Q$ is also a cubic Cremona transformation given in a similar form as $\Psi_Q$ such that the total transform of the base locus of the inverse of $\Psi_Q$ is the quadric $Q$. We get 
$$
\Bbb P^4\backslash (W\cup Q) \cong \Bbb P^4\backslash (P_W\cup Q').
$$
Varying the quadric hypersurface $Q$, we can cover $\Bbb P^4\backslash W$ by open sets of the form $\Bbb P^4\backslash (W\cup Q)$. 

Step 2: 
Now, we will cover the open set $\Bbb P^4\backslash (W\cup Q)$ by open $\Bbb A^2$-cylinders. Therefore, we perform a second quadric Cremona transformation. Let $q\in P_W\cap Q'$ be an intersection point and let $\mathcal{I}_q(1) = (l_1',l_2',l_3',l_4')$ be the linear forms generating its ideal. Let $Q' = V(g')$ be generated by a quadratic equation $g'$.  Let $P' = V(h')$ be the tangent hyperplane to $Q'$ at the point $q$.
Then, we consider the following quadric Cremona transformation 
$$
\Psi_{Q,q}:\xymatrix{
\Bbb P^4 \ar@{-->}[rrr]^{|h' \cdot l_1', h' \cdot l'_2, h' \cdot l_3', h' \cdot l_4', g'|} &&& \Bbb P^4.}
$$
The image of $Q'$ is a hyperplane $P_1\subset \Bbb P^4$.  Furthermore, the total transform of the base locus of $\Psi_{Q,q}$ is another hyperplane $P_2$.  The image of $P_W$ is a third hyperplane $P_3$.  Note that the inverse of $\Psi_{Q,q}$ is again a  quadric Cremona transformation whose total transform of its base locus is the tangent hyperplane $P'$.  We get 
$$
 \Bbb P^4\backslash (P'\cup P_W\cup Q')\cong \Bbb P^4\backslash (P_1\cup P_2\cup P_3)
$$

Let $\ell=P_1\cap P_2\cap P_3$ be the intersection. 

Then the projection $\pi_\ell: \Bbb P^4 \dashrightarrow  \Bbb P^{2}$  with center $\ell$ defines an $\Bbb A^2$-cylinder on $\Bbb P^4\backslash (P_1\cup P_2\cup P_3)$.
Varying again the point $q$ and therefore, the tangent hyperplane $P'$, we can cover $\Bbb P^4\backslash (P_W\cup Q')$ by $\Bbb A^2$-cylinders. 

Therefore,  $\Bbb P^4\backslash W$ can be covered by $\Bbb A^2$-cylinders of the form $\Bbb P^4\backslash (P\cup Q\cup W)$ where $P$ is the preimage of the hyperplane $P'$ as in Step 2.  Note that $P$ is again a hyperplane since $P'$ contains the vertex of $Q'$.  
\end{proof}

\begin{cor}\label{MukaiFourfoldA2cylinder}
 Any Mukai fourfold of genus $7,8$ and $9$ is $\Bbb A^2$-cylindrical. 
\end{cor}

\section{On the flexibility of Mukai fourfolds of $g= 7$, $8$, $9$}

We recall the definition of a transversal cover of cylinders and construct a transversal $\Bbb A^2$-cylinder covering of the complement $X_{2g-2}\backslash Y$ of a Mukai fourfold of genus $g$ and a hypersurface $Y$. Using the computer algebra system {\it Macaulay2} (see \cite{macaulay2}), we conclude that the general Mukai fourfold is flexible. 

\begin{defn}
\begin{enumerate}[$i)$]
\item (\cite[Definitions 3]{Per13})
A subset $Y\subset X$ is called \emph{invariant} with respect to a $\Bbb A^n$-cylinder $(Z, \varphi)$ in $X$ if $$Y\cap \varphi( Z\times \Bbb A^n)=\varphi(\pi_1^{-1}(\pi_1(\varphi^{-1}(Y)))),$$ where $\pi_1:Z\times \Bbb A^n\to Z$ is the first projection of the direct product. 
\item (\cite[Definitions 4]{Per13})
We say that a variety $X$ is {\it transversally covered} by $\Bbb A^n$-cylinders ${\{(Z_i, \varphi_i)\}_{i\in I}}$ in $X$ if 
$X$ has a  covering 
\begin{equation}\label{covering}
 X=\bigcup\limits_{i\in I} U_i
 \end{equation}
where each $U_i$ is a Zariski open subset in $X$ such that $U_i= \varphi_i(Z_i\times \Bbb A^n)$ and it does not admit any proper invariant subset with respect to  every $\Bbb A^n$-cylinders $(Z_i, \varphi_i)$. 
 A subset $Y\subset X$ is proper if it is nonempty and different from $X$. 
\end{enumerate}
\end{defn}

\subsection{Transversality of the $\Bbb A^2$-cylinder covering}

We choose hypersurfaces $Y$ such that the Mukai fourfold $X_{2g-2}$ and a collection of hyperplanes $P$ and quadric hypersurfaces $Q$ as in Lemma \ref{cylinderCremona} such that $X_{2g-2}\backslash Y\stackrel{\Phi}{\cong} \Bbb P^4\backslash W$ is covered by open $\Bbb A^2$-cylinders $V_{Q,P} := \Phi^{-1}(U_{Q,P})$ where 
$$U_{Q,P} := \Bbb P^4\backslash (P\cup Q\cup W) \cong \Bbb P^4\backslash (P^1_{Q,P}\cup P^2_{Q,P}\cup P^3_{Q,P}).$$
 
We recall some notation from the proof of Lemma \ref{cylinderCremona} that we need in the proof of the transversality of the covering $\{V_{Q,P}\}_{P,Q}$ of $X_{2g-2}\backslash Y$.

For the open subset $V_{Q,P}$, let $W$ be the singular cubic hypersurface in $\Bbb P^4$, and let $Q$ be a rank $4$ quadric hypersurface of $\Bbb P^4$ with vertex in a node of $W$.  There exists a hyperplane $P_{W}$ and a quadric $Q^\prime$ together with a cubic Cremona transformation, denoted by $\Psi_Q:\Bbb P^4 \dashrightarrow  \Bbb P^{4}$ such that 
   $$\Bbb P^4\backslash (W\cup Q)\cong \Bbb P^4\backslash (P_{W}\cup Q^\prime).$$ 
For the hyperplane $P$,  let $P' = \Psi_Q(P)$ be its image, there exist three hyperplanes $\{P^j_{Q,P}\}_{j=1}^3$ and
 a quadric Cremona transformation, denoted by $\Psi_{Q,P}:\Bbb P^4 \dashrightarrow  \Bbb P^{4}$ such that 
  $$\Bbb P^4\backslash ( P_W\cup Q^\prime\cup P^\prime)\cong \Bbb P^4\backslash (P^1_{Q,P}\cup P^2_{Q,P}\cup P^3_{Q,P})$$ as in Lemma  \ref{cylinderCremona}.  Let $\ell_{Q,P}=P^1_{Q,P}\cap P^2_{Q,P}\cap P^3_{Q,P}$ be the intersection.  
Then the structure of an $\Bbb A^2$-cylinder on $V^i_{Q,P}$ is induced by the projection $\pi_{Q,P}:\Bbb P^4 \dashrightarrow  \Bbb P^{2}$ with center $\ell_{Q,P}$. In summary, we have the following chain of birational maps
$$
\xymatrix{
X_{2g-2}  \ar@{-->}[r]^{\Phi} & \Bbb P^4 \ar@{-->}[r]^{\Psi_Q} & \Bbb P^4 \ar@{-->}[r]^{\Psi_{Q,P}} & \Bbb P^4 \ar@{-->}[r]^{\pi_{Q,P}} & \Bbb P^2.}
$$

\begin{prop}  
 The $\Bbb A^2$-cylinder covering of $X_{2g-2}\backslash Y$ as above is transversal. 
 \end{prop}
\begin{proof} Suppose that it is not transversal.  Let $v\in V$ be a point on a proper subset of $V\subset X_{2g-2}$ that is invariant with respect to our covering. 
For a cylinder $V_{Q,P}$ containing the point $v$, we replace $V$ by $V\cap V_{Q,P}$ and assume that $V$ is contained in the chosen cylinder.  
Using the above notation, we consider the $\Bbb A^2$-fiber of the cylinder $V_{Q,P}$ containing the image of the point $v$, that is, 
$$
S_0:= (\pi_{Q,P})^{-1}((\pi_{Q,P} \circ \Psi_{Q,P} \circ \Psi_Q \circ \Phi)(v))
$$ 
Let $S$ be the closure of the preimage $\Psi_{Q,P}^{-1}(S_0)$ which is a quadric surface since $\Psi_{Q,P}$ is a quadric Cremona transformation.  Note that $S$ is contained in $ (\Psi_Q \circ \Phi)(V)$ since $V$ is invariant with respect to $V_{Q,P}$. 

\begin{claim}
There exist another cylinder $V_{Q,P_1}$ such that the closure of the image 
$$C:=(\pi_{Q,P_1}\circ \Psi_{Q,P_1})(S)$$ is a line in $\Bbb P^2$.  
\end{claim}

\begin{proof}[Proof of the claim.]
Note that the intersection of the surface $S$ with the hyperplane $P_W$ is exactly the base locus of $\Psi_{Q,P}$ since $S_0$ contains the projection center $\ell_{Q,P}$ of $\pi_{Q,P}$. 

But the surface $S$ does not contain the base locus of $\Psi_{Q,P_1}$ for another cylinder $V_{Q,P_1}$ and therefore, the image of $S$ under $\Psi_{Q,P_1}$ is again a quadric surface that does not contain the projection center $\ell_{Q,P_1}$ of $\pi_{Q,P_1}$. Hence, $(\pi_{Q,P_1}\circ \Psi_{Q,P_1})(S)$ coincides with the image of the linear span $\langle\Psi_{Q,P_1}(S)\rangle\cong \Bbb P^3$ which is a line as claimed. 
\end{proof}

Therefore,  the closure of the preimage 
$$
S_1:=(\pi_{Q,P_1}\circ \Psi_{Q,P_1} \circ \Psi_Q)^{-1}(C)
$$
is a sextic hypersurface in $\Bbb P^4$ since $ \Psi_{Q,P_1} \circ \Psi_Q$  is the composition of a cubic and a quadric Cremona transformation and is again contained in $\Phi_i(V)$ by the invariance of $V$.

\begin{claim}
There exist another cylinder $V_{Q_2,P_2}$ such that 
$$
S_2 := (\pi_{Q_2,P_2}\circ \Psi_{Q_2,P_2} \circ \Psi_{Q_2})(S_1)
$$
is dense in $\Bbb P^2$.  
\end{claim}

\begin{proof}[Proof of the claim.]  
Notice that the intersection of the hypersurface $S_1$ and the quadric $Q$ are six planes passing through the node of $Q$. Furthermore the hypersurface $\Psi_Q(S_2)$ contains the base locus of $\Psi_{Q,P_1}$. But for general choice of $Q_2$ and $P_2$, the hypersurface $S_1$ and its image do not contain the base locus of $\Psi_{Q_2}$ and $\Psi_{Q_2,P_2}$, respectively. Hence, the image of $S_1$ under the composition $\Psi_{Q_2,P_2} \circ \Psi_{Q_2}$ for another cylinder $V^i_{Q_2,P_2}$ is a hypersurface of degree $24$. Hence, the claim follows. 
\end{proof}

Therefore, the preimage $(\pi_{Q_2,P_2}\circ \Psi_{Q_2,P_2} \circ \Psi_{Q_2})^{-1}(S_2)$ is dense in $X_{2g-2}$ and is contained the $\Phi(V)$. Since we assumed that $V$ is a proper subset, this is a contradiction. 
\end{proof}

\subsection{Generic flexibility of Mukai fourfolds of genus $7,8$ and $9$}\label{sectionGenericFlexibility}

\begin{defn}[\cite{AFKKZ13}] Let $V$ be an affine algebraic variety.  A point $p \in V$ is called {\it flexible} if the tangent space
$T_pV$ is spanned by the tangent vectors to the orbits of actions of the additive group of
the field $\Bbb G_a$ on $V$ . The variety $V$ is called {\it flexible} if each smooth point of $V$
is flexible . We also say that $V$ is {\it generically flexible} if every point in a non-empty Zariski open subset of $V$
is flexible.

Now, we recall the flexibility criterion for affine cones.

\end{defn}

\begin{thm}[\cite{Per13}]\label{flexibilityCriterionPer}
If for some very ample divisor $H$ on a smooth projective variety $X$ there exists a transversal covering by $H$-polar $\Bbb A^1$-cylinders,
then the affine cone over $X$ is flexible. 
\end{thm}

\begin{lem}\label{flexibilityCriterion} If $X$ is {\it transversally covered} by $\Bbb A^2$-cylinders, then $X$ is {\it transversally covered} by $\Bbb A^1$-cylinders.
In particular,  if for some very ample divisor $H$ on a smooth projective variety $X$ there exists a transversal covering by $H$-polar $\Bbb A^2$-cylinders,
then the affine cone over $X$ is flexible. 

\end{lem}
\begin{proof}
Assume that $X$ has a  transversally covering $ \{U_i\mid i\in I\}$, where each $U_i$ is a Zariski open subset in $X$ such that $U_i= \varphi_i(\Bbb A^2\times Z_i)$. For each $(i,j)\in I\times \Bbb A^1$, let $ U'_{i,j}=U_i$, $Z'_{i,j}=Z_i\times\Bbb A^1$, and $\psi_j:\Bbb A^1\times \Bbb A^1\to \Bbb A^2,(x,y)\mapsto (x+j,y+j)$. Then $\varphi'_{i,j}:Z'_{i,j}\times\Bbb A^1 \to A$ is defined as follow
$$\xymatrix{
Z_i\times\Bbb A^1\times \Bbb A^1=Z_{i,j}\times\Bbb A^1 \ar[d]_{\text{id}\times\psi_j}\ar[rrrr]^{\varphi'_{i,j}}&&&&X\\
Z_i\times\Bbb A^2\ar[urrrr]^{\varphi_{i}}&&&
}$$ 
Therefore $\{U'_{i,j}\mid (i,j)\in I\times \Bbb A^1\}$ is transversally covering by $\Bbb A^1$-cylinders $\{(Z'_{i,j},{\varphi'}_{i,j})\mid (i,j)\in\Bbb A^1\}$ in $X$.
\end{proof}

\begin{claim}\label{coverX}
Let $X_{2g-2}$ be a general Mukai fourfold of genus  $g=7$, $8$ and $9$.  Then
there exists an open covering $\{U_x\mid x\in \mathcal I\}$ of $X_{2g-2}$ such that the following conditions hold true.
\begin{enumerate}[$i)$]
\item $U_x=X_{2g-2}\backslash C_x$, where $C_x$ is a cubic hypersurface section of $X_{2g-2}$.
\item There exists  a singular cubic threefold $W_x\subset\Bbb P^4$,  and a birational map $\Phi_x:X_{2g-2} \dashrightarrow  \Bbb P^{4}$ such that  $$X_{2g-2}\backslash C_x\cong \Bbb P^4\backslash W_x.$$
\end{enumerate}
\end{claim}

\begin{proof}
Let $\mathcal I=\{x\in X\mid \text{ there are exactly five lines in } X \text{ passing through } x\}$. Then $\mathcal I$ is  a open set of $X$. By Proposition \ref{invMain}, given a point $x\in \mathcal I$, there are a unique cubic hypersurface section $C_x$ of $X_{2g-2}$, a singular cubic threefold $W_x\subset\Bbb P^4$,  and a birational map $\Phi_x:X_{2g-2} \dashrightarrow  \Bbb P^{4}$ such that  $$X_{2g-2}\backslash C_x\cong \Bbb P^4\backslash W_x.$$ 
Then we claim that 
$$\bigcap\limits_{x\in \mathcal I}C_x=\emptyset$$ 
holds for a general Mukai fourfold $X_{2g-2}$ which implies that $\{U_x\mid x\in \mathcal I\}$ as above is an open covering. It is sufficient to check this vanishing condition for one example, since it is an open condition in the moduli space of Mukai fourfolds of genus $g$. This computation is done in our ancillary {\it  Macaulay2} file (see \cite{HT22}).
\end{proof}

\begin{thm}\label{general}
The affine cones over general Fano-Mukai fourfold of genus  $g=7$, $8$ and $9$ are flexible.
\end{thm}
\begin{proof}
It follows from  Lemma \ref{flexibilityCriterion}, Claim \ref{coverX} and Corollary \ref{MukaiFourfoldA2cylinder}.
\end{proof}

Now, let us consider the generic flexibility of affine cones over Fano-Mukai fourfolds of genus $g=7$, $8$ and $9$.
 Then we have the following result.

\begin{thm}\label{generic}
Let $X_{2g-2}$ be a Fano-Mukai fourfold of genus $g=7$, $8$ and $9$. Then the affine cone over $X_{2g-2}$ is generically  flexible for any polarization of $X_{2g-2}$.
\end{thm}
\begin{proof}
Notice that for any ample divisor $H$ on $X_{2g-2}$ all the cylinders $V_{Q,P}$ are $H$-polar since $\mathrm{Pic}(X_{2g-2})=\Bbb Z$.
The statement follows from the above transversal cylinder cover of the complement of the hypersurface $Y$ in $X_{2g-2}$ and \cite[Theorem 2.4]{Per21}.
\end{proof}

The affine cones over any smooth Fano-Mukai fourfold of genus $10$ are flexible (\cite[Theorem A]{PrZ20}). An essential ingredient to show the flexibility in this case was the full description of the Hilbert scheme of cubic surfaces inside such a Mukai fourfold. A special point in this Hilbert scheme can be obtained as a special tangential intersection of a Mukai fourfold of genus $10$, namely, a cone over a twisted cubic (see Section \ref{lines}). To generalize Theorem \ref{generic} and \ref{general}, we conjecture that one also needs descriptions of other Hilbert schemes of surfaces in Mukai fourfolds of genus $7, 8$ and $9$ arising from special tangential intersections. 

\begin{prob}\rm
The affine cones over any smooth Fano-Mukai fourfold of genus  $g\in \{7,8,9\}$ are flexible.
\end{prob}


\bibliographystyle{amsalpha}

\end{document}